\documentclass[12pt]{amsproc}

\usepackage{amsmath, amsthm, amssymb, amsfonts, mathrsfs, amscd, tikz, bbm}
\usepackage{enumerate, paralist, hyperref, color, ulem}
\usetikzlibrary{matrix,arrows} 
\def\AA{{\mathbb A}}

\def\CC{{\mathbb C}}

\def\GG{{\mathbb G }}

\def\QQ{{\mathbb Q}}
\def\PP{{\mathbb P}}
\def\QQ{{\mathbb Q}}
\def\RR{{\mathbb R}}

\def\ZZ{{\mathbb Z}}

\def\0{{\mathbf 0}}
\def\1{{\mathbf 1}}

\def\Acal{{\mathcal A}}

\def\Lcal{{\mathcal L}}
\def\Mcal{{\mathcal M}}

\def\Ocal{{\mathcal O}}

\def\Scal{{\mathcal S}}

\def\Aut{\mathrm{Aut}}

\def\Cl{\mathrm{Cl}}
\def\codim{\mathrm{codim}}

\def\Div{\mathrm{Div}}
\def\diff{\mathrm{d}}

\def\div{\mathrm{div}}

\def\SL{\mathrm{SL}}
\def\PSL{\mathrm{PSL}}

\def\Rat{\mathrm{Rat}}

\def\PGL{\mathrm{PGL}}
\def\SO{\mathrm{SO}}
\def\GL{\mathrm{GL}}

\def\Pic{\mathrm{Pic}}

\def\Fix{\mathrm{Fix}}

\def\Res{\mathrm{Res}}
\def\Hom{\mathrm{Hom^n_d}}

\def\supp{\mathrm{supp}}

\def\min{\mathrm{min}}
\def\uf{\mathrm{uf}}
\def\uc{\mathrm{uc}}

\def\Ker{\mathrm{Ker}}

\def\Ratfd2{\mathrm{Rat}_{d,2}^{\uf}}
\def\Mfd2{\mathcal{M}_{d,2}^{\uf}}
\def\Rfd2{\mathcal{R}_{d,2}^{\uf}}
\def\Rcd2{\mathcal{R}_{d,2}^{\uc}}

\def\Md{\mathcal{M}_d}
\def\Mds{\mathcal{M}^s_d}
\def\Mdss{\mathcal{M}^{ss}_d}
\def\PPs{(\PP^{2d+1})^s}
\def\PPss{(\PP^{2d+1})^{ss}}
\def\Ratd{\mathrm{Rat_d}}

\theoremstyle{plain}

\newtheorem{thm}{Theorem}
\newtheorem*{thm*}{Theorem}

\newtheorem{cor}[thm]{Corollary}
\newtheorem{prop}[thm]{Proposition}
\newtheorem{lem}[thm]{Lemma}
\newtheorem*{prop*}{Proposition}

\theoremstyle{definition}
\newtheorem{dfn}{Definition}

\newtheorem{rem}{Remark}
\newtheorem{ex}{Example}
\newtheorem{notation}{Notation}
\begin{document}

\title[Automorphism Loci for the Moduli Space of Rational Maps]{Automorphism loci for the moduli space of rational maps}

\author{Nikita Miasnikov}
\author{Brian Stout}
\author{Phillip Williams}

\address{Nikita Miasnikov; Ph.D. Program in Mathematics; CUNY Graduate Center; 365 Fifth Avenue; New York, NY 10016 U.S.A.}
\email{nmiasnikov@gc.cuny.edu}

\address{Brian Stout; Department of Mathematics; United States Naval Academy;  Annapolis, MD 21401 U.S.A.}
\email{stout@usna.edu}

\address{Phillip Williams; Department of Mathematics; The King's College; New York, NY 10004 U.S.A.}
\email{pwilliams@tkc.edu}

\thanks{{\em Date of last revision:} June 26, 2014}

\begin{abstract}
Let $k$ be an algebraically closed field of characteristic $0$ and $\Md$ the moduli space of rational maps on $\PP^1$ of degree $d$ over $k$.  This paper describes the automorphism loci $A\subset \Ratd$ and $\mathcal{A}\subset \Md$ and the singular locus  $\mathcal{S}\subset\Md$. In particular, we determine which groups occur as subgroups of the automorphism group of some $[\phi]\in\Md$ for a given $d$ and calculate the dimension of the locus. Next, we prove an analogous theorem to the Rauch-Popp-Oort characterization of singular points on the moduli scheme for curves. The results concerning these distinguished loci are used to compute the Picard and class groups of $\Md, \Mds,$ and $\Mdss$. 
\end{abstract}

\maketitle

\section{Introduction and Main Results}\label{Introduction}
Let $k$ be an algebraically closed field and $\PP^1$ the projective line over $k$.  If one fixes homogeneous coordinates $X,Y$ on $\PP^1$, then any rational map $\phi:\PP^1\rightarrow\PP^1$ of degree $d>1$ can be realized as a pair of homogeneous polynomials of degree $d$ in $X$ and $Y$
\begin{equation*}
\phi(X,Y)=[F(X,Y):G(X,Y)]
\end{equation*}
where $F,G$ have no common roots in $k$. 

Fix a degree $d>1$.  The collection of all such pairs of homogeneous polynomials $[F:G]$ can be naturally parametrized as the projective space $\PP^{2d+1}$ where
\begin{equation*}
[F:G]=[a_0X^d+a_1X^{d-1}Y+\cdots+a_dY^d:b_0X^d+b_1X^{d-1}Y+\cdots+b_dY^d]
\end{equation*}
and
\begin{equation*}
[F:G]\mapsto [a_0:a_1:\cdots:b_d]
\end{equation*}
It is clear that not every such pair determines a rational map of degree $d$ on the projective line.  There is a homogeneous polynomial of degree $2d$ defined over $\ZZ$ in the coefficients $a_i,b_j$, called the {\it Macaulay resultant}, which vanishes precisely when $F$ and $G$ have a common root in $k$.  We denote the Macaulay resultant by $\Res$ and note that $\Res\in\Gamma(\PP^{2d+1},\Ocal(2d))$. It follows that we can construct the parameter space of rational maps on $\PP^1$ as the complement of the vanishing locus of $\Res$.

\begin{dfn}
The space of rational maps of degree $d$ on $\PP^1$ is $\Rat_d=\PP^{2d+1}-V(\Res)$.
\end{dfn}

From elementary algebraic geometry, we can see that $\Rat_d$ is an affine variety over $k$ of dimension $2d+1$.  We will often identify rational maps $\phi$ of degree $d$ with points of $\Rat_d$.

The space of rational maps carries a natural $\PGL_2$ action of conjugation.  For any $f\in\PGL_2$ and any rational map $\phi\in\Rat_d$ we define the {\it conjugate} of $\phi$ as $\phi^f=f^{-1}\circ\phi\circ f$. For technical reasons, we restrict ourselves to the analogous action of $\SL_2$ on $\Rat_d$.  Since $k$ is algebraically closed, $\mathrm{PSL}_2\cong\PGL_2$ and we loose no geometric information.

\begin{dfn}
The moduli space of rational maps of degree $d$ on $\PP^1$ is defined to be the geometric quotient $\Md=\Rat_d/\SL_2$. We use $\pi:\Rat_d\rightarrow\Md$ to denote the quotient map.
\end{dfn}

If two rational maps $\phi,\psi\in\Ratd$ are conjugate we say they are isomorphic. They have the same dynamics as rational maps on $\PP^1$.

The moduli space $\Md$ exists as an affine variety over $k$ of dimension $2d-1$.  This space is a coarse moduli space for the moduli problem of rational maps: its points parametrize conjugacy classes of rational maps.  This space has been studied as a complex orbifold in dynamics over $\CC$ (see Appendix G of \cite{milnor}) and Silverman studied this space as a geometric quotient scheme over $\ZZ$ (see \cite{silverman:space}). Moduli spaces for dynamics on higher dimension projective spaces were constructed and studied by Levy in \cite{levy:space}.

Little is known about the geometry of these spaces. Many, however, hope that by studying the geometry of these spaces that properties uniform for all dynamical systems can be determined, much in the same way that the modular curves $X_0(p),X_1(p)$, which are certain moduli of elliptic curves, were used to prove Mazur's Theorem on uniform boundedness of torsion on elliptic curves defined over $\QQ$.

We give a brief synopsis of the known results.  For dynamical systems of degree $2$, Silverman proved in \cite{silverman:space} that $\Mcal_2\cong\AA^2$ as a scheme over $\ZZ$. Silverman gives an explicit isomorphism map in terms of symmetric functions of the multipliers of the fixed points.  Levy has shown in \cite{levy:space} that all $\Md$ are rational varieties.  By geometric invariant theory, $\Md$ is known to be normal, integral, and connected.

As mentioned previously, the moduli space $\Md$ is not fine: families of dynamical systems parametrized by a base scheme $T$ do not correspond to morphisms $T\rightarrow\Md$.  One cause of this phenomena is the existence of rational maps with automorphisms.

\begin{dfn}
Let $\phi\in\Ratd$ be a rational map of degree $d>1$.  We say that $f\in\PGL_2$ is an {\it automorphism of $\phi$} if $\phi^f=\phi$.  
\end{dfn}

The collection of automorphisms for fixed $\phi$ forms a group, denoted by $\Aut(\phi)$; this group is the stabilizer of the point $\phi$ for the $\PGL_2$ action of conjugation. It is also clear that if $\phi,\psi\in\Ratd$ are isomorphic, then their automorphism groups are also isomorphic.  It is known that these groups are finite and are bounded in terms of $d$ (see \cite{levy:space}). For a point $[\phi]\in\Md$ we can define $\Aut([\phi])$ as an abstract group using $\Aut(\phi)$ for any map in the conjugacy class.

A natural question to ask is when a rational map has a non-trivial automorphism and to describe the automorphism locus in $\Md$.  We denote the locus of automorphisms in $\Ratd$ by $A$ and the corresponding locus in $\Md$ by $\Acal$.  Specifically,\begin{equation*}A=\lbrace\phi\in\Rat_d | \Aut(\phi)\neq 1\rbrace\end{equation*} and $\Acal=\pi(A)$. The only explicitly known case is for quadratic rational maps.  When $d=2$, via Silverman's isomorphism $\Mcal_2\cong\AA^2$, the automorphism locus $\Acal$ is known to be a cuspidal cubic. Furthermore, the only possible automorphism groups are $\Scal_3$, which occurs precisely at the cusp, and $\ZZ/2\ZZ$ which occurs at all other points of this curve (see Proposition 4.15 of \cite{silverman:msad}). For a general $d>1$, it is known that the only possible automorphisms groups are $\ZZ/m\ZZ$, the cyclic group of order $m$, $D_m$, the dihedral group of order $2m$, $A_4, A_5$, the alternating group on four or five letters, respectively, or $S_4$ the full symmetric group on four letters. See Example 2.54 of \cite{silverman:msad}.

A second special locus to consider is the singular locus.  We denote the singular locus of $\Md$ by $\Scal$, and $\Scal$ consists of all points in $\Md$ which have tangent space of dimension strictly larger that $2d-1$.  Again, the only explicitly known case is for quadratic rational maps, where $\Scal=\emptyset$ since $\Mcal_2$ is smooth.  It would be interesting to determine if these "singular" rational maps have additional dynamical structure.

A third geometric question about the moduli space of rational maps, which is related to the previous two, consists of calculating the Picard group $\Pic(\Md)$. When the action of an algebraic group on an affine variety is free, the Picard group of the quotient variety is particularly easy to calculate.  Because of the existence of non-trivial automorphisms, the action of $\SL_2$ on $\Ratd$ is not free.  However, after determining the automorphism loci $A\subset\Ratd$ and $\Acal \subset \Md$ to be of codimension greater than $1$, we compute $\Pic(\Md),\Pic(\Mds),$ and $\Pic(\Mdss)$  by looking at the locus where the action is free.

We now discuss the results of this paper.

We begin by studying the loci of points in $\Md$ whose automorphism group contains a copy of a cyclic or dihedral group.  

\begin{dfn}
Let $G$ be a finite subgroup of $\PGL_2$. \begin{equation*}\Acal(G)=\lbrace [\phi] |\Aut([\phi])\supseteq G\rbrace \end{equation*}
\end{dfn} 

In Section~\ref{AutomorphismLocus} we calculate the dimensions of $\Acal(G)$ when $G$ is cyclic or dihedral, in particular we show that $\Acal(\ZZ/m\ZZ)\neq \emptyset$ if and only if $m$ divides one of $d$, $d\pm1$.

In Section~\ref{platonic} we use the decomposition of L. West for the $2d+2$-dimensional representation of $\SL_2$ associated with conjugation of rational maps to obtain the following result which characterizes the remaining possible automorphism groups.
 
\begin{prop*} Let $d>1$. 
\begin{enumerate} 
\item  $\Acal(A_4) \neq \emptyset$ if and only if $d$ is odd.
\item  $\Acal(S_4) \neq \emptyset$ if and only if $d$ is coprime to $6$.
\item  $\Acal(A_5) \neq \emptyset$ if and only if $d$ is congruent to one of $1,11, 19, 21$ modulo $30$.
\end{enumerate}
Let $G$ be any of the groups $A_4, A_5, S_4$ and suppose $d$ is such that $\Acal(G)\neq\emptyset$.  Then $\Acal(G)$ is irreducible and 
 \[\dim\Acal(G)=\left[ \frac{2d}{|G|}\right].\]
 where $\left[ n\right]$ denotes the the greatest integer less than or equal to $n$.
\end{prop*}

In Section~\ref{AutSing} we apply the Luna Slice theorem and our computations of dimensions in Section~\ref{AutomorphismLocus} to prove the following theorem inspired by Rauch-Popp-Oort's characterization of singular points on the moduli space for curves.

\begin{thm*}
Let $d>2$. Then the singular locus $\Scal$ and the automorphism locus $\Acal$ of $\Md$ coincide.
\end{thm*}

The Luna Slice theorem is an algebro-geometric tool and for us $\Md$ is a quotient {\it variety}.   However, $\Md(\CC)$ equipped with the analytic topology becomes a topological quotient of $\Rat_d(\CC)$, when the latter is equipped with the usual topology (cf~\cite{neeman}).

Our strategy is as follows:
We call a point $\phi$ in the automorphism locus {\it simple} if $\Aut(\phi)\cong \ZZ/p\ZZ$ with $p$ prime.  We then show the following:

\begin{enumerate}
\item Simple points in $\Md$ are algebraically singular and in fact topologically singular, i.e. they do not possess a Euclidean neighborhood in the analytic topology.

\item Simple points  are dense in the automorphism locus of $\Md$.

\item Finally we use the fact that algebraic singularities form a closed set. 
\end{enumerate}

In Section~\ref{PicardGroups} we use the results of the previous sections to calculate Picard and class groups of $\Md$, its projective closure $\Mdss$ and the intermediate variety $\Mds$.  Our main tool for the following theorem is  Narahimhan-Drezet's Descent Lemma for vector bundles equipped with group action (see \cite{drezet:picard}). 

\begin{thm*}
Let $d>1$ and $\Md,\Mds,\Mdss$ be the moduli space of rational maps of degree $d$,  stable points in $\PP^{2d+1}$, and semi-stable points in $\PP^{2d+1}$.  Then,
\begin{enumerate}
\item $\Cl(\Md)=\ZZ/2d\ZZ\text{ or }\ZZ/d\ZZ$ when $d$ is odd and even, respectively.
\item $\Cl(\Mds)=\Cl(\Mdss)=\ZZ$.
\item $\Pic(\Md)$ is trivial.
\item $\Pic(\Mds)=\Pic(\Mdss)=\ZZ$.
\end{enumerate}
\end{thm*}
{\it Acknowledgment.}   The authors would like to thank John Milnor for his helpful correspondence and Lloyd West for helpful conversations.


\section{Dimension of the automorphism locus}\label{AutomorphismLocus}
In this section we work over $\CC$, although the conclusions hold for any algebraically closed field of characteristic $0$.  

We recall the definition of an automorphism of a rational map from Section~\ref{Introduction}.

\begin{dfn}
Let $\phi\in\Ratd$ be a rational map of degree $d>1$.  We say that $f\in\PGL_2$ is an {\it automorphism} if $\phi^f=\phi$.  
\end{dfn}

For every rational function $\phi$ there is a canonical metric $g$ on $\PP^1(\CC)$ which makes $\PP^1(\CC)$ isometric to the usual $2$-dimensional sphere.  Any automorphism $f\in\Aut(\phi)$ must fix the metric $g$, and therefore $\Aut(\phi)$ embeds into $\SO_3(\RR)$. It is well known that $\Aut ( \phi )$ is always finite (Proposition 4.65 of \cite{silverman:ads}) and the complete list of finite subgroups of $\SO_3(\RR)$ are:  cyclic order $m$, dihedral $D_m$ (of order $2m$), and the groups of rotations of platonic solids $A_4$ (tetrahedron), $S_4$ (cube and octahedron), and $A_5$ (dodeca- and icosohedron) (Remark 4.66 \cite{silverman:ads}).  

Thus the groups arising as $\Aut (\phi)$ are all from this list.    In this section we examine $\Acal(G)$ and calculate its dimension for all $G$.  We will use the shorter notation $\Acal_m= \Acal(\ZZ/m\ZZ)$ for the case of a cyclic group of order $m$.

We start with a well-known lemma, and then proceed case-by-case depending on the subgroup $G\subset\PGL_2$.

\begin{lem}\label{FiniteSubgropsOfPGL2} Two finite subgroups of $\PGL_2$ which are isomorphic are conjugate.  Moreover,
\[  \left<\zeta_m z\right>\cong\ZZ/m\ZZ, \;\;\;\;\left<\zeta_mz, 1/z\right>\cong D_m,\]
\[\left< -z, i\frac{z+1}{z-1} \right>\cong A_4, \;\;\;\;\left< iz, i\frac{z+1}{z-1} \right>\cong S_4.\]

\end{lem}

\begin{proof}
The first statement is well known and dates back to Klein, \cite{klein}, and the rest follows immediately. 
\end{proof}

\subsection{Cyclic groups}
First we look at the case of a cyclic group.  The presentation of the material in this section was suggested by John Milnor.  Let $\zeta$ be a primitive $m$-th root of unity with $m>1$.  By Lemma \ref{FiniteSubgropsOfPGL2} every finite cyclic subgroup of $\PGL_2(\CC)$ of order $m$ is conjugate to $\left<z\mapsto \zeta z\right>$.  If $\phi$ commutes with $\left<z\mapsto \zeta z\right>$ then it must map the set $\{ 0,\infty\}= \Fix( \left<z\mapsto \zeta z\right>)$ to itself. Hence, $\phi$ induces a set map $\{0,\infty\}\rightarrow\{0,\infty\}$, leading us to the following definition suggested  again by John Milnor.

\begin{dfn} If $\sigma\in\PGL_2(\CC)$ is an non-trivial automorphism of a rational function $\phi$,  we say that $\sigma$ is an automorphism of $\phi$ {\it of type }$t$ if 
 \[t+1 = |\Fix (\phi) \cap \Fix (\sigma) |. \]
\end{dfn}

Since every automorphism of $\PP^1$ of finite order is conjugate to $\zeta z$, it has two fixed points and hence the type $t$ can only be $-1,0$, or $1$. Let $A_\zeta^t\subset\Ratd$ (where $t$ is abbreviated to one of $-,0,+$ rather than $-1,0,1$) be the locus of all maps for which $ \left<z\mapsto \zeta z\right>$ is an automorphism of type $t$.  We will now characterize $A_\zeta^t$.

\begin{lem} \label{commutes} A rational map $\phi$ commutes with the linear map $z\mapsto \zeta z$ if and only if it has the form $\phi(z)=z\psi(z^m)$ for some rational function $\psi$.
\end{lem}

\begin{proof} Setting $\eta(z) = \phi(z)/z$, note that
\[\eta(\zeta z) = \phi(\zeta z)/(\zeta z) = \phi(z)/z = \eta(z).\]
It follows easily that $\eta(z)= \psi(z^m)$ for a uniquely defined rational function $\psi$.  The converse is a straightforward calculation.
\end{proof}

Let $d\geq 2$ be the degree of $\phi$ and let $d'\geq 1$ be the degree of $\psi$.  By Lemma \ref{commutes}, set 
\[\psi(u) = \frac{\alpha u^{d'} +\dots+ \beta}{\gamma u^{d'}+\dots+ \delta}\]
so that 
\[\psi(\infty) = \alpha/\gamma \mbox{    and    } \psi(0) = \beta/\delta.\]
It follows easily that
\[\phi(0) = \infty \iff \psi(0) =\infty \iff \delta=0,\]
with $\phi(0) =0$ whenever $\delta\neq 0$.  Similarly
\[\phi(\infty) = 0 \iff \psi(\infty) =0  \iff \alpha=0,\]
with $\phi(\infty)=\infty$ whenever $\alpha\neq 0$.  

The discussion can be now be divided into three cases according to the type of the map $z\mapsto\zeta z$  as an automorphism of $\phi$.

{\bf Type 1.} If $\alpha\neq 0$, and $\delta\neq 0$, so that $0$ and $\infty$ are fixed points of $\phi$, then it follows easily that 
\[d= md'+1,\]
\[\dim (A_\zeta^+) =2d'+1.\]

{\bf Type -1.} If  $\alpha = \delta =0$, so that $\{0,\infty\}$ is a period two orbit for $\phi$, then canceling a factor of $z$ from the numerator and denominator of $\phi$ we see that $d<md'$, and it follows easily that 
\[d=md'-1,\]
\[\dim (A_\zeta^-) =2d'-1.\]

{\bf Type 0.}  In the intermediate case where just one of the two coefficients $\alpha$ and $\delta$ is zero, so that $\phi(0) = \phi(\infty)\in\{0,\infty\}$, the locus $A^0_\zeta$ has two irreducible components which are conjugate via $z\mapsto 1/z$ : 
\[\{\phi| \phi(\zeta z) = \zeta\phi(z), \phi(0)=\phi(\infty)= 0\}\subset\Ratd \]
and
\[ \{\phi| \phi(\zeta z) = \zeta\phi(z), \phi(0)=\phi(\infty)= \infty\}\subset\Ratd \] 
and a similar argument shows that for each component,
\[d=md',\]
\[\dim (A_\zeta^0) =2d'\]

Let $\Acal_m^t=\pi(A^t_m)$ be the locus in $\Md$ of points parametrizing conjugacy classes which admit an automorphism of order $m$ and type $t$.  Then 
\[\dim (\Acal_m^t)=\dim (A_\zeta^t)-1,\]
 since the group $\{z\mapsto cz| c\in k\}$ acts on $A_\zeta^t$ and all fibers of the surjective map $A_\zeta^t\to\Acal_m^t$ are $1$-dimensional, being finite unions of orbits under this action.   

This discussion leads to the following proposition.

\begin{prop}\label{prop:dim}
The space $\Rat_d$ admits a map with an automorphism of order $m$ of type $t$  if and only if $m$ is a divisor of $d-t$.  In such a case the resulting locus in $\Md$ of all maps having an automorphism of order $m$ and type $t$ is closed and irreducible of dimension $$\dim(\Acal^t_m)=\dfrac{2(d-t)}{m}+t-1$$.
\end{prop}

\begin{proof}
Everything has been established above except that the locus is closed.  This follows from Luna Slice theorem (cf~\cite{drezet:lunaslice}).
\end{proof}

\begin{rem}
Except for the case $m=2$ and $d$ odd, an automorphism of order $m$ can occur only in one type. In this case $\Acal_m=\Acal^t_m$ depending on the type $t$ that does occur.  When $d$ is odd, automorphisms of order $2$ can occur in types $+1, -1$.  Thus for odd $d$, $\Acal_2= \Acal_2^- \cup \Acal_2^+$. 
\end{rem}
\begin{rem}\label{rem:dim}
It is easy to see that if $2<m<n$, 
\[d-1=\dim (\Acal_2)>\dim (\Acal_m) \geq \dim (\Acal_n).\]
\end{rem}

\begin{cor} \label{codimension} For $d\geq 1$, the codimension of the automorphism loci $A \subset \Ratd$ and $\mathcal{A} \subset \Md$ is $d-1$.
\end{cor} 

\begin{proof}
This follows immediately for $\mathcal{A} \subset \Md$. For $A$, we have that because $A$ is $\SL_2$ invariant, closed, and contained in the stable locus, $\mathcal{A}$ is a geometric quotient of $A$ by the action of $\SL_2$. It follows that $\codim(A,\Ratd)=\codim(\Acal,\Md)$. 
\end{proof}

\subsection{Dihedral groups}
We now examine the locus, $\Acal(D_m)$, of those $[\phi]$ for which $\Aut( \phi)$ contains a copy of the dihedral group $D_m$ of order $2m$.   Recall, that by Lemma \ref{FiniteSubgropsOfPGL2}, every dihedral group $D_m$  inside $\PGL_2(\CC)$ is conjugate to $\left<\zeta_mz, 1/z\right>$.   

\begin{lem}\label{DihedralCommuteLemma} 
A rational function $\phi$ commutes with both $\zeta_mz$ and $1/z$ if and only it can be written in the form
$\phi= z\psi(z^m)$ where $\psi$ commutes with $1/z$.  
\end{lem}

\begin{proof}
By Lemma \ref{commutes} $\phi$ can be written in the form $\phi(z) = z\psi(z^m)$ if and only if it commutes with $\zeta_m z$. Now, $z \psi (z^m)$ commutes with $1/z$ if and only if $z\psi(z^m) = \frac{1}{1/ z\psi{(1/z^m)}}  \iff \phi(z^m) = \frac{1}{\psi(1/z^m)} \iff \psi(u) = \frac{1}{\psi(1/u)}$ if and only if $1/u$ commutes with $\psi$. 
\end{proof}

\begin{prop}\label{prop:dihedral}
The dimensions of $\Acal(D_m)$ are as follows:
\begin{enumerate} 
\item[($t=1$)] if $m|d-1$, $\dim \Acal(D_m) = \frac{d-1}{m}$
\item[($t=0$)] if $m| d$, $\Acal(D_m) =\emptyset$
\item [($t=-1$)] if $m|d+1$, $\dim \Acal(D_m) = \frac{d+1}{m}-1$
\end{enumerate}
\end{prop}
\begin{proof}
Suppose $\phi$ commutes with both $\zeta_m$ and $1/z$. Then we have from Lemma \ref{DihedralCommuteLemma} that $\phi = z\psi(z^m)$ where $\psi$ commutes with $1/z$. Let $[a_0, \dots, a_{d'}, b_0, \dots, b_{d'}]$ be the coefficients of $\psi$. 

In the $t=1$ case, the degree of $\psi$ is $d' = \frac{d-1}{m}$, and this gives $2d'+2=2\frac{d-1}{m} +2$ coefficients. Because $\psi$ commutes with $1/z$, the coefficients of $\psi$ satisfy the additional $d'+1$ homogeneous equations $a_0 = \lambda b_{d'}, a_1 = \lambda b_{d'-1}, \dots, a_{d'} = \lambda b_0$ where $\lambda=\pm 1$. Each of these independent equations cuts the dimension down by one. Therefore the dimension of the locus in $\Rat_d$ in this case is: $$2d'+2 -(d'+1) -1 = d' = \frac{d-1}{m}$$

In the $t= 0$ case, exactly one of $a_0, b_{d'}$ vanishes, which is impossible for a map that commutes with $1/z$. Therefore, $\Acal(D_m) =\emptyset$. 

In the  $t=-1$ case, the degree of $\psi$ is $d' = \frac{d+1}{m}$, and this gives $2d'+2=2\frac{d+1}{m} +2$ coefficients. That $1/z$ is an automorphism of $\psi$ gives the same relations as in the $t=1$ case; also we have that $a_0 = 0, b_{d'} = 0$, so the first relation $a_0 =\pm b_{d'}$ is redundant. This gives a total of $d'+2$ independent equations. Therefore the dimension of the locus in $\Rat_d$ in this case is $$2d'+2 - (d'+2) -1= \frac{d+1}{m}-1$$

Since there is no infinite family in $\PGL_2$ which commutes with both $\zeta_m z$ and $1/z$, the map from the locus $A(D_m)$ in $\Rat_d$ to $\mathcal{A}(D_m)$ is finite to one, and the result follows. 
\end{proof}

\begin{rem}Note that for $m=2$ the cases $(\pm1)$ agree with each other.
\end{rem}

\subsection{Automorphism locus in \texorpdfstring{$\PP^{2d+1}$}{Projective Space}}
\begin{prop}\label{codimautproj}
For $d\geq 1$ the codimension of the automorphism locus in $\PP^{2d+1}$ is $d-1$.
\end{prop}

\begin{proof} One can prove this directly in a way similar to the proof for $\Rat_d$.  However here is another proof.  We have already proved the corresponding statement for  $\Rat_d$ (Corollary \ref{codimension}), so for any irreducible component $Z$ which meets $\Ratd$, $\codim Z\geq d-1$.  Let $Z$ be an irreducible component of $\Aut(\PP^{2d+1})$ entirely inside $V(\Res)$.  Let 
\[m= \min\{\deg \gcd(F,G) | (F:G) \in Z\}\] 
and let $W=H^0(\PP^1, O(m))$, the space of $m$-forms in two variables. Consider the morphism:
\[f:\PP(W)\times \PP^{2(d-m)+1}\to V(\Res)\]
given by 
\[\big( [H], ( F: G)\big)\mapsto ( HF : HG )\]

It is not hard to see that $\PP(W)\times\Rat_{d-m}$ carries a $\PGL_2$-action and that the map $f$ is equivariant for this action. Since the image of $f$ restricted to $\PP(W)\times\Rat_{d-m}$ consists of all pairs $(F:G)$ with $\deg(\gcd(F,G))= m$ (i.e. common zeros of exact order $m$), it follows that$f^{-1}(f(\PP(W)\times\Rat_{d-m})) = \PP(W)\times\Rat_{d-m}$.

We claim the restriction of $f$ to $\PP(W)\times \Rat_{d-m}$ is injective. Suppose $\big( [H_1], ( F_1: G_1)\big), \big( [H_2], ( F_2: G_2)\big)$ both map to $(H_1F_1:H_1G_1)=(H_2F_2:H_2G_2)$. The, for some $\lambda\in k^\times$, $H_1F_1=\lambda H_2F_2$ and $H_1G_1=\lambda H_2G_2$. These equalities imply that $H_1 | H_2F_2$ and $H_1 | H_2G_2$, but $(F_2:G_2)\in\Rat_{d-m}$ means they have no common zeros.  It follows that $H_1 | H_2$, and since their degrees are the same, $[H_1]=[H_2]$ in $\PP(W)$. After canceling in the previous equations we have that $(F_1:G_1)=(F_2:G_2)$ and the restriction is injective.

Now, the image of $f$ contains $Z$ and $Z$ meets $f(\PP(W)\times \Rat_{d-m})$.  From injectivity it follows that 
\[f(\Aut (\PP(W)\times \Rat_{d-m})) = \Aut (f(\PP(W)\times \Rat_{d-m}))\]
and we also have 
\[\Aut (\PP(W)\times \Rat_{d-m})\subseteq \PP(W)\times \Aut(\Rat_{d-m})\]
Thus we get
\begin{equation*}
\begin{split}
\dim Z &= \dim Z\cap f(\PP(W)\times\Rat_{d-m})\\
&\leq \dim\Aut f(\PP(W)\times\Rat_{d-m})\\
&\leq\dim\Aut(\PP(W)\times\Rat_{d-m})\\
&\leq\dim \PP(W)\times (\Aut\Rat_{d-m})\\
&= \begin{cases} m+(d-m)+2= d+2 &\mbox{for } m<d \\ 
d+1 & \mbox{for }m=d \end{cases} \\
\end{split}
\end{equation*}
We remark that the first equality arises because $f$ is projective, and hence closed, and so $\overline{f(\PP(W)\times\Rat_{d-m})}=f(\PP(W)\times\PP^{2(d-m)+1})$.
\end{proof}

Given an element $g\in\SL_2$, a point $\phi\in\PP^{2d+1}$ fixed by $g$ comes from an eigenvector of multiplication by $g$ in $\AA^{2d+2}$. Specifically, $g$ acts on the stalk of $\Ocal(1)$ and $\Ocal(-1)$ at $\phi$ via multiplication by the eigenvalue.

\begin{lem} \label{eigenvalues} Let $\eta\neq 0$.  Suppose $\Phi=(F,G)\in\AA^{2d+2}$ is an eigenvector for the linear transformation given by the action of $ \left( \begin{array}{cc}
\eta & 0\\
0 & \eta^{-1} \end{array} \right)$ and $\lambda$ the associated eigenvalue.  Write 
\[(F,G)=(a_0,\dots,a_d, b_0,\dots,b_d).\]
Then,
\begin{compactenum} 
\item If $k$ is such that $a_k\neq 0$, then $\lambda= \eta^{d-2k-1}$.
\item If $k$ is such that $b_k\neq 0$, then $\lambda= \eta^{d-2k+1}$.
\end{compactenum}
\end{lem}

\begin{proof}
Under the hypothesis of the lemma it follows that:
\begin{equation*}
(\eta^{-1} F(\eta X,\eta^{-1} Y),\eta G(\eta X,\eta^{-1} Y))=\lambda(F,G)
\end{equation*}
and therefore that
\begin{equation*}
(a_0\eta^{d-1},\ldots,\dfrac{a_d}{\eta^{d+1}},b_0\eta^{d+1},\ldots,\dfrac{b_d}{\eta^{d-1}})=(\lambda a_0,\ldots,\lambda a_d,\lambda b_0,\ldots ,\lambda b_d)
\end{equation*}
If $a_k\neq 0$, then one can equate the corresponding coordinates and solve for $\lambda$, resulting in the first equation. In a similar fashion one obtains the equation when $b_k\neq 0$.
\end{proof}

\begin{cor}\label{eigenRatd}
Suppose $\phi\in\Ratd$ and $\sigma\in\PGL_2$ an automorphism of $\phi$ of order $m > 1$ and type $t$.  Let $d'= (d-t)/m$.  Let 
$g\in\SL_2$ represent $\sigma$  and assume that order of $g$ is $2m$.  Then $g$ acts on the stalk of $\Ocal(1)$ at $\phi$ via multiplication by a root of unity whose order $s$ is as follows:
\begin{compactenum}
\item[{\bf Case $t=\pm1$:}]  $s=1$ if $d'$ is even, and $s=2$ if $d'$ is odd.
\item[{\bf Case $t=\;\;\;0$:}]  $s=m$ if $d$ is odd and $s=2m$ if $d$ is even.
\end{compactenum}
\end{cor}
\begin{proof}  After an appropriate conjugation we can assume that \begin{equation*}\sigma(z)=\zeta z,\end{equation*} where $\zeta$ is $m$-th root of unity.  Then we have the following:  If $t=1$, $a_0\neq 0$.  If $t=-1$, $b_0\neq 0$.  If $t=0$,  either $a_0\neq 0$ or $b_0\neq 0$, with the same result.
\end{proof}

The above lemma and corollary will help us to compute the Picard and class groups of $\Md$, $\Mds$, and $\Mdss$ in Section $\ref{PicardGroups}$. 
\section{Automorphism loci for groups of rotations of platonic solids}\label{platonic}
In this section we use L. West's decomposition to calculate dimension formulas for automorphism loci for tetrahedral group $A_4$, octahedral group $S_4$ and the icosahedral group $A_5$.   

Let $V$ be the vector space $k^2$.  Then $\SL_2$ acts naturally on the left on $V$ by multiplication.  Let $X,Y$ be the coordinate functions on $k^2$ corresponding to the standard basis, and let $R=k[X,Y]$.  Let $R_n$ denote the $n$-th homogeneous component of $R$.  Since $R$ is the coordinate ring of $V$, the action of $\SL_2$ on $V$ induces a right $\SL_2$-action on $R$ by pre-composition, where if
\begin{equation*}
g=\left( \begin{array}{cc}
a & b \\
c & d \end{array} \right)
\end{equation*}
and $F\in R$, then $F^g=F(aX+bY,cX+dY)$.  This action preserves grading hence there is a (right) action on each homogeneous component $R_n$.  Let $W$ be the space of morphisms $V\to V$ given by a pair of homogeneous polynomials of degree $d$ in $X$ and $Y$.  Then $\SL_2$ acts on the right on $W$ by conjugation: $(\phi^A)(v) = A^{-1}\phi(Av)$.
\begin{thm}[L. West \cite{west}]   The following is an isomorphism of $\SL_2$ representations:
\[g: W\to R_{d-1}\oplus R_{d+1}\]
\[g(F_1,F_2) =\left(\frac{ \partial F_1}{\partial X}+\frac{ \partial F_2}{\partial Y}, YF_1-XF_2\right)\]
whose inverse is:
\[g^{-1}(H,J) = \frac{1}{d+1}\left(XH +\frac{ \partial J}{\partial Y}, YH-\frac{ \partial J}{\partial X}\right)\]
\end{thm}
\begin{proof} 
Morphisms $V\to V$ can be canonically (independently of choice of coordinates on $V$) identified with vector fields on $V$, by which we mean elements of $T:=\mbox{Hom}_R( \Omega^1_{R/k}, R)$.  Namely for each vector field $F= F_1\partial/\partial X+ F_2\partial/\partial Y\in T$, we get the morphism $V\to V$ given by
$(X,Y)\to (F_1, F_2)$.  Let us now endow $V$ with the volume form $\left<,\right>$ given by  $\left< (x_1,  y_1), (x_2, y_2)\right> = x_1 y_2 - y_1 x_2$.  This form is preserved by the $\SL_2$ action.
To prove that $g$ is an $\SL_2$-morphism, we will define $g$ in terms of quantities preserved by the $\SL_2$-action.
\[g(F) = (\nabla\cdot F, \left< F,r \right>)\]
\[g^{-1}(H,J) = \frac{1}{d+1} (Hr + (\diff J)^{\sharp})\]
where $r=(X,Y)$, $v^\flat=\left<v, -\right>$, $^\sharp$ is the inverse of $^\flat$, and where we identify vector fields on $k^2$ with maps $k^2\to k^2$.

\end{proof}
\begin{notation}  By $\bar{g}$ we denote the $\PGL_2$-isomorphism
\[\bar{g}: \PP^{2d+1}\to \PP(R_{d-1}\oplus R_{d+1})\]
which is induced by $g$.
\end{notation}

The following is a key lemma about the image of $\Ratd$ under the above decomposition $\bar{g}$. 
\begin{lem}\label{keylemma} Let $\GG_m$ act on $R_{d-1}\oplus R_{d+1}$ via 
\[t\cdot (H,J) = (tH, t^{-1}J).\]
This action induces a $\GG_m$-action on $\PP(R_{d-1}\oplus R_{d+1})$. This action commutes with the $\PGL_2$ action.   Hence $\PGL_2$-stabilizers are identical for all points within a $\GG_m$-orbit.  The $\GG_m$ orbit of $[(H,J)]$ meets $\bar{g}(\Ratd)$ if and only if no multiple zero of $J$ is also a zero of $H$.
\end{lem}

\begin{proof}
The commutativity with the $\PGL_2$ action is immediate. Suppose that a multiple zero of $J$ is also a zero of $H$. Without loss of generality, we may suppose this zero to be the coordinate $Y$, and we have $$g^{-1}(tH,t^{-1}J) = \dfrac{1}{d+1}\left(tXH + t^{-1}\dfrac{\partial J}{\partial Y}, tYH -  t^{-1}\dfrac{\partial J}{\partial X} \right)$$ for all $t$. From this formula it is clear that $Y$ is a common root for any $t$ and thus $\GG_M\cdot[(H,J)]$ does not meet $\bar{g}(\Ratd)$. 

Conversely, suppose $$[(F,G)] = \dfrac{1}{d+1}\left[\left(tXH + t^{-1}\dfrac{\partial J}{\partial Y}, tYH -  t^{-1}\dfrac{\partial J}{\partial X} \right)\right] \notin \Ratd$$ for all $t$. This means that $F$ and $G$ have a common zero for all $t$. Now, a well known identity says that $$X \dfrac{\partial J}{\partial X}+Y \dfrac{\partial J}{\partial Y} = (d+1)J$$

Thus if $(a,b)$ is a common zero of $F$ and $G$, then evaluating the above expression at $(a,b)$ gives that $t^2 aH = -\dfrac{\partial J}{\partial Y}|_{(a,b)}$ and  $t^2 bH = \dfrac{\partial J}{\partial X}|_{(a,b)}$. Plugging this into the above identity gives that $(a,b)$ is also a zero of $J$. Assuming $J$ is not the zero polynomial (in which case we have nothing to show), this implies that at least one of the common zeros of $F$ and $G$ occurs infinitely often as $t$ varies. Because the $\PGL_2$ action commutes with the $\GG_m$ action, we may take this zero to be $[0:1]$--i.e. $Y$ is a factor of $F$ and $G$ for infinitely many $t$. Now, if $H$ and $\dfrac{\partial J}{\partial Y}$ are not both divisible by $Y$, it would be impossible for $tXH + t^{-1}\dfrac{\partial J}{\partial Y}$ to be divisible by $Y$ for infinitely many $t$, and so we have that $Y$ divides both $H$ and $\dfrac{\partial J}{\partial Y}$. Similarly $Y$ must divide $\dfrac{\partial J}{\partial X}$. It follows that $Y$ is a multiple root of $J$. 
\end{proof}

Lemma \ref{keylemma} implies the following interesting corollary, which we will not need in the remainder of the paper.

\begin{cor} Let $G$ be non-cyclic and suppose that $\dim\Acal(G)=0$.  Then if $\Aut(\phi)$ contains a copy of $G$, then $\phi$ must have $\deg(\phi)+1$ distinct fixed points with all multipliers equal to $\deg(\phi)$.
\end{cor}

\begin{lem}\label{NoCommonZerosLem}
Let  $m, k$ be positive integers, $\eta$ a primitive $2m^{th}$ root of $1$, and $F$ a binary form of degree $k$ for which $0$ and $\infty$ are not multiple zeros and suppose $F$ is an eigenvector for the action of $ \left( \begin{array}{cc}
\eta & 0\\
0 & \eta^{-1} \end{array} \right)$ with eigenvalue $\lambda$.   Then $\lambda$ is as in Table~\ref{AZetaTable}.
\begin{center}
\begin{table}[h]\label{forms-eigen}
\begin{tabular}{|c||l|l|}
\hline
\multicolumn{1}{|l||}{Fixed zeros}   & \multicolumn{1}{|c}{${F(0)\neq0}$}                                                                                      
& \multicolumn{1}{|c|}{${F(0)=0}$}                                                                                \\ \hline\hline
${F(\infty)\neq 0}$ & \begin{tabular}[c]{@{}l@{}}$m|k$\\ $\lambda=(-1)^{\frac{k}{m}}$\end{tabular} &\begin{tabular}[c]{@{}l@{}}$m|k-1$\\ $\lambda=(-1)^\frac{k-1}{m}\eta$\end{tabular}                                                                                                                         \\ \hline
${F(\infty)=0}$      & \begin{tabular}[c]{@{}l@{}}$m|k-1$\\ $\lambda=(-1)^\frac{k-1}{m}\eta^{-1}$\end{tabular}             & \begin{tabular}[c]{@{}l@{}}$m|k-2$\\ $\lambda=(-1)^{\frac{k-2}{m}}$\end{tabular} \\ \hline
\end{tabular}
\caption{Eigenforms without $0$ and $\infty$ as multiple zeros}
\label{AZetaTable}\label{CodimOfAutLocus}
\end{table}
\end{center}
\end{lem}

\begin{proof}
Write $F(X,Y)=a_0X^k+a_1X^{k-1}Y+\cdots+a_kY^k$. Because $F$ is an eigenvector we have that
\begin{equation*}
\begin{split}
F(\eta X,\eta^{-1}Y)&=a_0\eta^kX^k+a_1\eta^{k-2}X^{k-1}Y+\cdots +\dfrac{a_k}{\eta^k}Y^k\\
&=\lambda a_0X^k+\lambda a_1X^{k-1}Y+\cdots+\lambda a_kY^k
\end{split}
\end{equation*}
We write $F(0):=F(0,1)$ and $F(\infty):=F(1,0)$.  If $F(0)=0$, then $a_k=0$, but because $0$ is not a multiple zero, $a_{k-1}\neq 0$.  Similarly, if $F(\infty)=0$, then $a_0=0$, but since $\infty$ is not a multiple zero, $a_1\neq 0$.

The four cases in the table now correspond to four systems of linear equations.  If $F(0),F(\infty)\neq 0$, then $a_0,a_k\neq0$.  Equating first and last coefficients and canceling $a_0$ and $a_k$ gives the system $\lambda=\eta^k$ and $\lambda=1/\eta^k$.  Since $\eta$ is a $2m^{th}$ root of unity, it follows that $m|k$ and $\lambda=(-1)^\frac{k}{m}$.

Similarly, if $F(0)=F(\infty)=0$, then $a_0=a_k=0$ and we can equate the second and second to last coefficients giving the system $\lambda=\eta^{k-2}$ and $\lambda=1/\eta^{k-2}$.  In this case $m|k-2$ and $\lambda=(-1)^\frac{k-2}{m}$.

The other two cases are symmetrical. Suppose $F(0)=0$ and $F(\infty)\neq 0$.  Then $a_k=0,a_{k-1}\neq 0$, and $a_0\neq 0$. Equating first and second to last coefficients and canceling $a_0$ and $a_{k-1}$ gives $\lambda=\eta^k$ and $\lambda=1/\eta^{k-2}$. Solving for $\eta$ gives $1=\eta^{2k-2}$, hence $m|(k-1)$ and $\lambda=(-1)^\frac{k-1}{m}\eta$. If $F(0)\neq 0, F(\infty)=0$ the calculation is identical, but $\lambda=(-1)^\frac{k-1}{m}\eta^{-1}$.
\end{proof}

As described in the beginning of this section, there is a natural $\SL_2$-action on $R_n$ which induces a $\PSL_2=\PGL_2$ action on $\PP(R_n)$, which we can consider as a $\PGL_2$ action on the group of divisors: if $F\in R_n$ and $D=\div(F)$, then $D^g=\div(F^g)$. Thus if $D=\sum n_i (P_i)$, $D^g=\sum n_i (g^{-1}\cdot P_i)$. 

For a divisor $D$ on $\PP^1$ and a finite group $G\subset\PGL_2$, then we can consider the orbit of $D$ under $G$.  If $p\in\PP^1$, then we can consider $p$ as a divisor and speak of its orbit under this action. This orbit consists of at most $|G|$ points. If the orbit has less that $|G|$ points, then $D$ has non-trivial stabilizer under this action.

\begin{lem}\label{InvariantEigenvalueLem}  
Let $G\subset \PGL_2$ be a finite subgroup of order $k$, $k$ even.  Let $ P \in\PP^1$.  Let $D$ be the divisor
\begin{equation*} D=\sum \limits_{\sigma\in G} P^\sigma
\end{equation*}
Let $\sigma_0\in G$ and let $g$ be a pre-image of $\sigma_0$ in $\SL_2$ and denote by $m$ the order of $\sigma_0$. Let $F \in R_k$ be such that $\div(F) = D$. Then the action of $g$ on $F$ is multiplication by $(-1)^\frac{k}{m}$, and in particular, if $G$ is one of $A_4$, $S_4$, $A_5$, the pre-image of $G$ in $\SL_2$ fixes $F$.
\end{lem}

\begin{proof}
Let $n$ be the order of the stabilizer of $P$ in $G$.  Then for any $\sigma\in G$, the stabilizer of $P^\sigma$ in $G$ is also order $n$, so $D=nD'$ for some divisor $D'$ with no coefficients greater than 1. $D'$ is, in fact, the support of $D$. The degree of $D'$ is $\frac{k}{n}$.

Now let $F$ be a homogeneous form such that $\div(F) = D$ and let $\sigma_0$ and $g$ be as in the statement of the lemma. We can write $F = H^n$ where $H$ is a homogenous form with no multiple zeros and $\div(H) = D'$. Since $D$ is clearly fixed by $G$, the action of $g$ on $F$ must be multiplication by a scalar $\lambda$. The same is true for $D'$ and $H$: $g$ acts on $H$ by a scalar $\lambda'$ and we have that $(\lambda')^n = \lambda$.

Now, with Lemma \ref{NoCommonZerosLem} in mind, let $\eta$ be a primitive $2m^{th}$ root of unity and consider the matrix: $$ \left( \begin{array}{cc}
\eta & 0\\
0 & \eta^{-1} \end{array} \right) \in \SL_2$$ 

Conjugating $g$ to this matrix does not change the eigenvalues $\lambda$ and $\lambda'$, and so without loss of generality we may assume $g$ has this form. Note also that the two lifts of $\sigma_0$ to $g\in\SL_2$ give the same eigenvalue $\lambda$ for $F$ because $k$ is even. Now apply Lemma \ref{NoCommonZerosLem} to $H$. For the top left case of this lemma, the eigenvalue of $H$ is $\lambda ' = (-1)^\frac{k/n}{m}$ so that $\lambda = (-1)^\frac{k}{m}$.  For the bottom right case, first note that in this case $0$ is a zero of $H$, and also a fixed point of $\sigma_0$ (after our conjugation). This means that $\sigma_0$ is in the stabilizer of $0$, which, being a root of $H$, is order $n$. Thus $m|n$. Then $(-1)^\frac{k/n-2}{m}$ raised to the $n^{th}$ power gives $\lambda = (-1)^\frac{k-2n}{m}=(-1)^\frac{k}{m}(-1)^\frac{-2n}{m} = (-1)^\frac{k}{m}$.  The other two cases are similar to this, yielding $\lambda = (-1)^\frac{k}{m}$, and we omit the calculation.

If $G=A_4,S_4,A_5$, then $k=12,24,60$, respectively.  For $A_4$, the possible $m$ are $m=2,3$, so $k/m=6,4$ and $(-1)^\frac{k}{m}=1$. The justification for $S_4,A_5$ is entirely analogous and we, again, omit the calculation.
\end{proof}

In the following definitions let $G\subset\PGL_2$ be a finite subgroup.

\begin{dfn} 
 We say that an orbit of $P\in\PP^1$ is {\it degenerate} if the length of the orbit $|GP|$ is less than the order of the group $|G|$.
\end{dfn}

\begin{dfn}\label{defRelPair} 
 We call a pair of effective divisors  on  $\PP^1$ $(D_1,D_2)$, where $D_1=\sum_p s_p p$ and $D_2=\sum_p t_p p$ a {\it relevant pair}  if:
\begin{compactenum}
\item $\deg D_2 - \deg D_1\equiv 2 \pmod{|G|}$ 
\item   If $H$, $J$ are  homogeneous forms corresponding to $D_1$, $D_2$ respectively, then for the action of every element of the pre-image of $G$ in $\SL_2$,  $H$ and $J$ are eigenvectors with equal eigenvalues.
\item For all $p\in\PP^1$, $t_p>1 \implies s_p=0$.
\item  For all $p\in\PP^1$, $s_p, t_p <|G_p|$. 
\end{compactenum}
\end{dfn}

\begin{dfn}
Similarly, we call an effective divisor $D=\sum t_p p$ a {\it relevant divisor} if:
\begin{compactenum}
\item  $D$ is fixed by $G$.
\item All $t_p$ are either $0$ or $1$.
\item Every point in $\supp(D)$ has a degenerate orbit.
\end{compactenum}
\end{dfn}

If $D$ is any effective divisor and $G$ a finite subgroup of $\PGL_2$ let $D_G$ be the effective divisor obtained from $H$ subtracting as many divisors of the form
\[\sum\limits_{\sigma\in G}P^\sigma\]
as possible.

Let $\bar{g}(\Ratd)^G= \bar{g}(\Ratd^G)$ be the fixed set for the action of $G$ on $\bar{g}(\Ratd)$.

\begin{lem}
Let $G$ be a finite subgroup of $\PGL_2$. Let $[(H,J)]$ be in $\bar{g}(\Ratd)^G$.  If $H\neq 0$, then
\begin{equation*}(\div(H)_G, \div(J)_G)\end{equation*} is a relevant pair. Furthermore, if $H=0$ then $\div(J)_G$ is a relevant divisor.
\end{lem}

\begin{proof}
For $H \neq 0$, (1) follows from the fact that $H$ is degree $d-1$, $J$ is degree $d+1$, and any divisors of the form $\sum\limits_{g\in G} gp$ that we subtract in passing to $(\div(H)_G, \div(J)_G)$ have degree $|G|$. For (2), being in $\bar{g}(\Ratd)^G$ implies that $H$ and $J$ are eigenvectors with equal eigenvalues for the preimage of $G$ in $\SL_2$, and then by Lemma \ref{InvariantEigenvalueLem} this is unchanged when removing the linear factors corresponding to the divisors that were subtracted by passing to $(\div(H)_G, \div(J)_G)$. The requirement in (3) follows immediately for $(\div(H), \div(J))$ by Lemma \ref{keylemma}, and remains true when passing to $(\div(H)_G, \div(J)_G)$ because any $t_p$ can only get smaller in this process. For (4), first note that the divisors $\div(H)$ and $\div(J)$ are fixed by the action of $G$. Divisors of the form $\sum\limits_{g\in G} gp$ are fixed by the action of $G$, and so when we remove them the resulting divisor pair $(\div(H)_G, \div(J)_G)$ is also fixed by $G$. Note also that if $p$ is in the support of a divisor fixed by $G$, then clearly the entire orbit of $p$ is in that support as well, since otherwise the divisor would change under the action of $G$. It is also clear that the coefficients in a divisor must be equal for all points in the same orbit. Thus if any of the coefficients $s_p$ or $t_p$ of $(\div(H)_G, \div(J)_G)$ are greater than or equal to $|G_p|$, it would be possible to subtract a divisor of the form $\sum\limits_{g\in G} gp$, contrary to the definition of $D_G$.

For $H=0$, (1) is clear and (2) follows as above from Lemma \ref{keylemma}. For (3), if a given $p$ does not have a degenerate orbit, then the divisor $\sum\limits_{g\in G} gp$ has coefficients all equal to $1$. Since, as above, the full orbit of $p$ occurs in the divisor $\div(J)$ with equal coefficients, it follows that the entire orbit will be subtracted off in passing to $\div(J)_G$, leaving only degenerate orbits.
\end{proof}

We have therefore have defined a map 
\begin{equation*}S:\bar{g}(\Ratd)^G\to \mbox{ relevant pairs } \cup\mbox{ relevant divisors.}\end{equation*}

\begin{prop}\label{RelevantPairProp} 
A relevant pair $(D_1,D_2)$ is in the image of $S$ if and only if 
 \begin{compactenum}
 \item  $\deg D_1 \leq d-1$ 
 \item $\deg D_2\leq d+1$, and $\deg D_2 \equiv d+1\pmod{|G|}$
 \end{compactenum}
\end{prop}

\begin{proof}
Let $(D_1,D_2)=(\div(H)_G,\div(J)_G)$ be in the image of $S$.  Then $\deg(D_1)=\deg(\div(H)_G)\leq\deg(\div(H))=d-1$ and similarly, $\deg(D_2)=\deg(\div(J)_G)\leq\deg(\div(J))=d+1$. Since $D_2$ is obtained by subtracting orbits of points from $\div(J)$, we have that $d+1=\deg(\div(J))=\deg(\div(J)_G)+m|G|$, so $\deg(D_2)\equiv d+1\pmod{|G|}$.

Conversely, let $(D_1,D_2)$ be a relevant pair of divisors.  We aim to construct a pair of forms $(H,J)$ of degree $d-1$ and $d+1$, respectively, such that \begin{equation*}(D_1,D_2)=(\div(H)_G,\div(J)_G).\end{equation*}
Let $H_0$ and $J_0$ be forms associated to $D_1$ and $D_2$.  Since $(D_1,D_2)$ is a relevant pair, we have that $\deg(J_0)=\deg(H_0)+2+m|G|$ for some $m$. We also have that, for every action of the pre-image of $G$ in $\SL_2$, $H_0$ and $J_0$ are eigenvectors with equal eigenvalues, and that every multiple zero of $J_0$ is not a zero of $H_0$, and finally that every zero of $H_0$ or $J_0$ is degenerate. By condition $(2)$ of this proposition, we have that $\deg(J_0)+\ell |G|=d+1$. Pick $\ell$ non-degenerate orbits of $\PP^1$ and let $J_1$ be the product of these orbits as homogeneous forms.  Set $J=J_0J_1$. It is clear that $\div(J)_G=J_0$. Similarly, $\deg(H_0)+n |G|=d-1$. Choose $n$ non-degenerate orbits of $\PP^1$ distinct from the previous $\ell$ ones.  Let $H_1$ be the product of these orbits as homogeneous forms and set $H=H_0H_1$. Again, $\div(H)_G=H_0$.

That $H$ and $J$ are $G$-invariant follows by condition (2) of a relevant pair and the construction of $H_1$ and $J_1$. By construction, any multiple zeros potentially added to $J$ will not also be zeros of $H$, so by Lemma \ref{keylemma}, $[(tH,t^{-1}J)]$ is in $\bar{g}(\Ratd)$ for all but finitely many $t$.
\end{proof}

We have a similar description of relevant divisors.

\begin{prop}\label{RelevantDivisorProp}
A relevant divisor $D$  is in the image of $S$ if and only if 
\begin{compactenum}
\item $\deg D\leq d+1$
\item $\deg D\equiv d+1\pmod{|G|}$
\end{compactenum}
\end{prop}

\begin{proof}
This proof is entirely analogous to the proof of Proposition \ref{RelevantPairProp} and therefore we omit it.
\end{proof}

Propositions \ref{RelevantPairProp} and \ref{RelevantDivisorProp} now allow us to calculate the dimension of fibers over a relevant pair $(D_1,D_2)$ or a relevant divisor $D$.

\begin{equation*}
\begin{split} 
\dim S^{-1}(D_1,D_2) &= \frac{d-1-\deg D_1}{|G|} +\frac{d+1-\deg D_2}{|G|} +1\\
&= \frac{2d-(\deg D_1+\deg D_2)}{|G|}+1\\
&=\left[\frac{2d}{|G|}\right] +1- \left[\frac{\deg D_1 +\deg D_2}{|G|}\right] 
\end{split}
\end{equation*} 
 
Similarly, for a relevant divisor $D$
\begin{equation*}
\dim S^{-1}(D) = \frac{d+1-\deg D}{|G| }
\end{equation*}

\begin{prop}
Fix a finite non-trivial subgroup $G\subset\PGL_2$. Then there are only finitely many relevant pairs and relevant divisors on $\PP^1$ for that group.
\end{prop}
\begin{proof}
Condition (4) in the definition of a relevant pair implies that a relevant pair only involves degenerates orbits, since if $G_p=1$, it follows that $s_p=t_p=0$. Degenerate orbits occur only for points that have non-trivial stabilizer, and there are clearly only finitely many such points in $\PP^1$. We have a bound on the size of the coefficient of each point in our pair of divisors, and so the result follows. The proof for relevant divisors is similar. 
\end{proof}

Since there are finitely many relevant pairs and relevant divisors, it is now clear how to calculate $\dim\Ratd^G$.

\begin{lem} \label{lemRelPairs}
Let $G$ be one of $A_4, S_4, A_4$.  Then there are $3$ degenerate orbits, $2^3=8$ relevant divisors and $8$ relevant pairs which are of the form  $(D_1(D_2),D_2)$  
where $D_2=\sum_p t_p p$ is a relevant divisor and $D_1(D_2) = \sum_p s_p p$ is defined by
\begin{equation*} s_p = \begin{cases} 0 &\mbox{if } t_p=1 \\ 
|G_p|-1 & \mbox{if } t_p=0. \end{cases} 
\end{equation*}

\end{lem}
\begin{proof}
The group of characters $C$ of $G$ is cyclic of order $3,2,1$ for $G=A_4, S_4, A_5$ resp.  Let $G$ be one of theses groups.   Let $\chi$ denote a generator for the group of characters.  Using geometry of a regular solid and Lemma~\ref{NoCommonZerosLem}, one sees that degenerate orbits of $G$ and their corresponding characters are as follows:
\begin{equation}
\begin{array}{ccc}
A_4 & S_4 & A_5 \\
 \begin{array}{cc}
  size & character \\
  4      & \chi\\
  6      & 1\\
  4      & \chi^{-1}
 \end{array} 
 
 &
 \begin{array}{cc}
  size & character \\
  8      & 1\\
  12      & \chi\\
  6      & \chi
 \end{array} 
  &
 \begin{array}{cc}
  size & character \\
  12      & 1\\
  30      & 1\\
  20      & 1
 \end{array} 

\end{array}
\end{equation}
Let $H$ be the quotient of $\Div^G$, the group of divisors on $\PP^1$ fixed by $G$, by the subgroup generated by divisors of the form $\sum\limits_{\sigma\in G}P^\sigma$.   $H$ is a finite group.  Each divisor fixed by $G$ gives rise to a character and this induces a surjective homomorphism
\[\alpha: H\to C\]
 Let $n=|G|$. The degree of the divisor gives rise to a homomorphism 
 \[\overline{\deg}:H\to 2\ZZ/n\ZZ\]
 Let $u$ be the image in $H$ of the divisor which is the sum of all points in $\PP^1$ with non-trivial stabilizers in $G$, with all multiplicities equal $1$:
 \[\sum\limits_{|G_P|\neq \left<1\right>} P \mapsto u\in H \]
  From the table above we see that $\deg u =n+2$ and that $\alpha(u)=1$.  The element $u$ shows that the resulting homomorphism
  \[H\to C\times 2\ZZ/n\ZZ\] 
  is surjective, while counting the orders of the groups shows that this homomorphism is in fact an isomorphism.   Pairs of divisors satisfying conditions 1), 2) and 4) of the definition of relevant pairs (Definition~\ref{defRelPair}) correspond bijectively to pairs of the form $(x, x+u)$ where  $x\in H$.  Imposing the remaining condition 3) on such pairs yields the desired result.
  
\end{proof}
\begin{thm}Let $d>1$. 
\begin{enumerate} 
\item  $\Acal(A_4) \neq \emptyset$ if and only if $d$ is odd.
\item  $\Acal(S_4) \neq \emptyset$ if and only if $d$ is coprime to $6$.
\item  $\Acal(A_5) \neq \emptyset$ if and only if $d$ is congruent to one of $1,11, 19, 29$ modulo $30$.
\end{enumerate}
Let $G$ be any of these groups and suppose $d$ is such that $\Acal(G)\neq\emptyset$.  Then $\Acal(G)$ is irreducible and
 \[\dim\Acal(G)=\left[ \frac{2d}{|G|}\right].\]
\end{thm}

\begin{proof}
 In the notation of Lemma~\ref{lemRelPairs}, $S^{-1}(D_2)\cup S^{-1}(D_1(D_2),D_2)$ is closed and irreducible in $\Ratd$.   For $S_4$ and $A_5$ the $8$ relevant divisors occur for distinct residues of $d$ modulo $|G|$, while for $A_4$ it is not the case, but the resulting multiple components are conjugate.    The total number of points in the three degenerate orbits of $G$ equals $|G|+2$.  With the aid of this fact, one deduces from the formula above that 
\begin{equation*}
\left[ \frac{\deg D_1(D_2) + \deg D_2}{|G|}\right]=1 
\end{equation*}
 for all relevant divisors.  That proves the proposition for $d>|G|$.  For $d\leq |G|$, the formula is easy to check.
 \end{proof}
  
 \begin{cor} Let $G$ be the group of rotations of a platonic solid, and $d>1$ such that $A(G)\neq\emptyset$.  Then $A(G)$ contains a function with distinct fixed points and all multipliers equal to $d+1$.
\end{cor}

\begin{ex}
 Let $d=5$.  Then $\Acal(A_4)= \Acal(S_4)$ is a single point in $\Mcal_5$ given by 
\[ f(z)=\frac{z^5-5z}{1-5z^4}.\]

 \end{ex}

\section{Automorphism locus equals singular locus}\label{AutSing}
In this section we combine Luna's slice theorem (cf Drezet \cite{drezet:lunaslice}) and the dimension calculations of Section~\ref{AutomorphismLocus} to show that the singular locus of $\Md$ coincides with its automorphism locus for $d>2$.  The impetus for the work presented in this section is an example of a family of singularities in $\Mcal_3$ in John Milnor's dynamics book (\cite{milnor}).  Effectively, in that example he identifies a Luna etale slice and uses it to find singularities.   The following is a variation on a definition suggested by John Milnor in personal correspondence(\cite{milnorEmails}).
\begin{dfn}  We call a point $[\phi]\in \Md(k)$  {\it simple} if $\Aut(\phi)$ is cyclic of prime order.
\end{dfn}

\begin{lem}
\label{rootaction}
Let  $p$ be a prime number and let $H$ be a cyclic group of order $p$, acting on  a vector space $V$.   Suppose
 $\codim(V^H,V)>1$.  Let $x\in V$.  The image of $x$ in the quotient  $V/H$ is a singular point precisely when the stabilizer $H_x$ is non-trivial.  
\end{lem}

\begin{proof}
The first part of this lemma is the Chevalley-Shephard-Todd theorem (see \cite{chevalley}).  We call $g\in H$ a {\it pseudoreflection} if it fixes a codimension 1 subspace of $V$ and is not the identity transformation. The Chevalley-Shephard-Todd theorem states that for arbitrary finite $H$ (i.e. not necessarily cyclic, prime order), $\bar{x}$ is non-singular if and only if $H_x$ is generated by pseudoreflections. When $H$ is finite cyclic with $\codim(V^H,V)>1$ there can be no pseudoreflections, so $x$ will be non-singular if and only if $H_x=1$.
\end{proof}

\begin{rem}
In fact,  over $\CC$, for such $x$, letting $\ell=2\dim V^H +2$, we have \[H_{\ell}( V/H, V/H -\bar{x} )= \ZZ/p\ZZ. \] This implies that the singularities of $\Md$ are topological, not merely algebraic.

To see this, first translate $x$ to $0$ and fix a basis of $V$ and choose the representation of $H$ where $H$ is the cyclic group generated by $g\in\GL_n$ with $g$ is diagonalized and diagonal entries given by $p^{th}$ roots of unity.

Write $V$ as a sum of the fixed subspace and its complement: $V=V^H+V^\perp$.   Then $V$ is a sum of $H$-invariant spaces $V^H$ and $V^\perp$ and the action of $H$ preserves the action of the norm $||\cdot ||$ on $V$ because the non-zero entries of each matrix in $H$ are roots of unity. Hence, the unit ball maps to itself under the action of $H$.  Let $B$ be the unit ball and $S$ the unit sphere in $V$. The unit ball is the cone over the sphere, and the action on the sphere is free, so the quotient map $S\rightarrow S/H$ is a covering map.

Let $C(X)$ denote the cone over a variety $X$ with $H$-action.  Then $C(X)/H=C(X/H)$.  Since the ball is the cone of the sphere $B/H=C(S)/H=C(S/H)$. Now consider the long exact sequence of relative homology for the cone and the punctured cone.  The punctured cone retracts to $X$ and the cone retracts to the vertex. Since our cone is the ball, the punctured cone retracts $S$ and the cone retracts to a point. Since the homology of a point is $0$, we have that $H_i(C(X),C(X)-x)=H_{i-1}(X)$ and so $H_i(B,B-x)=H_{i-1}(S)=0$ or $\ZZ$ if $i\neq n+1$ and $i=n+1$, respectively.

We now that $H_1$ is the abelianization of the fundamental group, and the fundamental group of $S/H$ is $\ZZ/p\ZZ$, so $H_1(S/H)=\ZZ/p\ZZ$ and therefore $H_2(B/H,B/H-\bar{x})=\ZZ/p\ZZ$. We now use excision.  Assume that $V^H=0$ Then $H_i(V/H,V/H-\bar{x})$ can be calculated by excising the outside of $B/H$ and we have that $H_i(V/H,V/H-\bar{x})=H_i(B/H,B/H-\bar{x})$, in particular $H_2(V/H,V/H-\bar{x})=\ZZ/p\ZZ$. If $V^H\neq 0$, then $H_{2\dim(V^H)+i}(V/H,V/H-\bar{x})=H_i(B/H,B/H-\bar{x})$ and we see that $H_\ell(V/H,V/H-\bar{x})=\ZZ/p\ZZ$ when $\ell=2\dim(V^H)+2$. We omit the last step of this justification as it takes us too far astray from our original goal.
\end{rem}

\begin{lem}Let $d>2$. Simple points of  $\Md(k)$ are singular.
\end{lem}
\begin{proof}
Let $\phi\in \Rat_d$, $H$ be the stabilizer of $\phi$, and assume $H =\left<\sigma\right>$ is cyclic of prime order $p$.  Let $N$ be the normal space at $\phi$ to the orbit of $\phi$
\begin{equation*}  
N = T_\phi(\Rat_d) / T_\phi(\SL_2\cdot \phi).
\end{equation*}
There is an induced action of $H$ on $N$.
Consequently by Luna's slice theorem for smooth varieties, the completion of the local ring at $[\phi]$ in $\Md$ is isomorphic to the completion of the local ring at the image of $0$ in the quotient $N/H$.    Thus we need only to prove that the image of $0$ in $N/H$ is a singular point.  It follows again from Luna's slice theorem for smooth varieties, that 
\begin{equation*}
\codim (N^H,N ) = \codim(\Acal_p,\Md)
\end{equation*}
The latter codimension is at least 2 for $d>2$, as shown in Section \ref{AutomorphismLocus}.  Therefore, by Lemma~\ref{rootaction}, $[\phi]$ is singular in $\Md$.
\end{proof}

\begin{rem}[{Simple points are topologically singular}] If $\pi: \Rat_d\to \Md$ is the quotient map, then $\pi^{an}: \Rat_d(\CC)\to\Md(\CC)$ is a topological quotient with respect to the classical topology ( cf Neeman,~\cite{neeman}).  Moreover,  isomorphism of completions of local rings in the proof above implies that in classical topology, there is a neighborhood of $[\phi]$ in $\Md$ which is homeomorphic to a neighborhood of the image of $0$ in $N/H$.  The image of $0$ in $N/H$ is topologically singular according to lemma~\ref{rootaction}.
\end{rem}

\begin{lem}\label{SimplePointsDense}
Simple points are dense in the automorphism locus $\Acal$.
\end{lem}

\begin{proof}  
It suffices to show that simple points are dense in $\Acal_p$ for each prime divisor $p$ of $(d-1)d(d+1)$.   Start with  $\Acal_2$.  Recall that $D_m$ denotes the dihedral group of order $2m$ and consider 
\begin{equation*}
\Acal_2\backslash (\bigcup_{m>2} \Acal_m\cup \Acal(D_2)).
\end{equation*}  
It is a dense subset of $\Acal_2$  because $\dim\Acal_m<\dim\Acal_2$ for $m>2$ (see Remark~\ref{rem:dim}) and $\dim\Acal(D_2) < \dim\Acal_2$ by Proposition~\ref{prop:dihedral}.  It contains only points with stabilizer isomorphic to $\ZZ/2\ZZ$, since every platonic solid has a symmetry of order $3$.   

Let now $p$ be a prime greater than $2$.  If $\Acal_p\subseteq\Acal_2$, simple points with stabilizer isomorphic to $\ZZ/2\ZZ$ are already dense in $A_p$.  Otherwise 
\begin{equation*}
\Acal_p\backslash(\Acal_2\cup\Acal_{2p}\cup\Acal_{3p}\cup\dots)
\end{equation*}
is non-empty, open, and consists of points with stabilizer isomorphic to $\ZZ/p\ZZ$.  It is non-empty because $\dim\Acal_{kp}<\dim\Acal_p$ for $k>1$ as seen from the formula in Proposition~\ref{prop:dim}.  It consists of simple points because dihedral groups and symmetry groups of platonic solids contain an element of order $2$.
\end{proof}

\begin{ex}
In the case of degree $2$ maps, $\Acal_3 = \Acal(D_3)$, hence simple points are not dense in $\Acal_3$. This agrees with the well known result that $\Mcal_2$ is smooth, but has non trivial automorphism locus.
\end{ex}

\begin{thm} \label{propsingaut}
When $d>2$, the singular locus of $\Md$  coincides with the automorphism locus: $\Acal=\Scal$.
\end{thm}
\begin{proof}
The singular locus is closed and Lemma \ref{SimplePointsDense} shows that $\Scal\cap\Acal$ is dense in the automorphism locus, hence $\Acal$ is contained in the singular locus $\Scal$.  The converse follows directly from Luna's slice theorem. We conclude that $\Acal=\Scal$. 
\end{proof}

\section{The Picard and class groups of \texorpdfstring{$\Md, \Mds,\text{ and }\Mdss$}{moduli space}}\label{PicardGroups}

In this section we use the results of Section \ref{AutomorphismLocus} and \ref{AutSing} to compute the Picard and class groups of $\Md$, $\Mds$, and $\Mdss$, where $\Mds$ and $\Mdss$ denote the moduli space of stable and semistable rational maps in the sense of Geometric Invariant Theory (GIT).  In this section we work over an algebraically closed field $k$ of characteristic $0$. We will freely use terminology from GIT, for which we refer the reader to \cite{mumfordGIT}.

\begin{lem}
Fix $d>1$ an integer.  Then $\Pic(\Ratd)\cong\ZZ/2d\ZZ$.
\end{lem}

\begin{proof}
As $\Ratd=\PP^n-V(\Res)$ is an affine open subscheme of projective space it immediately follows that $\Pic(\Ratd)\cong\Cl(\Ratd)$ and it suffices to show that $\Cl(\Ratd)\cong\ZZ/2d\ZZ$.

The resultant $\Res$ is a homogeneous polynomial of degree $2d$ on $\PP^{2d+1}$. It follows from Proposition 6.5 of \cite{hartshorne} that the following sequence is exact.
\begin{equation*}
\ZZ\rightarrow\Cl(\PP^{2d+1})\rightarrow\Cl(\Ratd)\rightarrow 0
\end{equation*}
where the first map is defined by $1\mapsto 1\cdot V(\Res)$. Using the well known fact that $\Cl(\PP^n)\cong\ZZ$ for any $n$, we conclude that $\Cl(\Ratd)\cong\ZZ/\deg(\Res)\ZZ=\ZZ/2d\ZZ$.
\end{proof}

In fact, this lemma proves something stronger, that $\Pic(\Ratd)$ is cyclic order $2d$ generated by $\Ocal(1)|_{\Ratd}$.

We must next define some terms from GIT.

\begin{dfn}
For $G$ a group variety acting on a variety $X$ via the action $\sigma:G\times X\rightarrow X$ and for any invertible sheaf $\Lcal\in\Pic(X)$ a {\it G-linearization} of $\Lcal$ is an isomorphism $\sigma^*\Lcal\rightarrow p^*_2\Lcal$, where $p_2$ is the projection $G\times X\rightarrow X$, subject to a certain cocycle condition which is equivalent to requiring that the corresponding line bundle is equipped with a $G$-action which is compatible with the $G$-action on $X$.
\end{dfn}

The collection of $G$-linearized line bundles on $X$ form a group which we denote by $\Pic^G(X)$. If we fix a $G$-linearized line sheaf on the variety $X$ we can then describe the largest subset of $X$ which has a geometric quotient.

\begin{dfn}
Let $G$ be an algebraic group acting on a variety $X$, all over $k$, and let $\Lcal$ be a $G$-linearized line sheaf.  Let $x\in X$ be a geometric point.  We call $x$ {\it semi-stable} if there exists a section $s\in H^0(X,\Lcal^n)$ for some $n$, such that $s(x)\neq 0$, the fiber $X_s$ is affine, $s$ is invariant. If, furthermore, the action of $G$ on $X_s$ is closed, then we call $x$ {\it stable}.   
\end{dfn}

We let $X^s$ denote the stable locus for the action of $G$ on $X$ and fixed linearization, and $X^{ss}$ the semi-stable locus.  It is clear from the definition that these are open subsets.  In this paper, $(\PP^{2d+1})^s$ and $(\PP^{2d+1})^{ss}$ refer to the stable and semi-stable locus for the $\SL_2$-linearization of $\Ocal(1)$ on $\PP^{2d+1}$ for the $\SL_2$-action of conjugation.

\begin{lem}\label{picard}
Fix $d>1$ an integer and let $\PPs$ and $\PPss$ denote the stable and semi-stable locus for the $\SL_2$-action of conjugation on $\PP^{2d+1}$.  Then the open immersions $\PPs\subset\PPss\subset\PP^{2d+1}$ induce isomorphisms $\Pic(\PPs)\cong\Pic(\PPss)\cong\Pic\PP^{2d+1}\cong\ZZ$.
\end{lem}

\begin{proof}
Let $Z^s=\PP^{2d+1}-\PPs$ and $Z^{ss}=\PP^{2d+1}-\PPss$ be the unstable, and not semi-stable loci, respectively.  As the stable and semi-stable loci are open, it follows that $Z^s$ and $Z^{ss}$ are closed subsets of $\PP^{2d+1}$.  Furthermore, $Z^{s},Z^{ss}\subset V(\Res)$ are proper closed subsets of the irreducible closed subset $V(\Res)$. Consequently $\codim(Z^s,\PP^{2d+1})>1$ and $\codim(Z^{ss},\PP^{2d+1})>1$.  It follows from Proposition 6.5 of \cite{hartshorne} that $\Pic(\PPs)\cong\Pic(\PP^{2d+1})\cong\ZZ$ and similarly $\Pic(\PPss)\cong\ZZ$.
\end{proof}

Let $A\subset\PP^{2d+1}$ be the subset of points of $\PP^{2d+1}$ which have a non-trivial stabilizer for the $\PGL_2$-action of conjugation.   $A$ is closed in $\PPss$ as a consequence of Luna's slice theorem.

\begin{dfn}
The set $\PP^{2d+1}_*=\PP^{2d+1}-A$ denotes the {\it free locus} of $\PP^{2d+1}$ for the $\PGL_2$-action of conjugation.    For any subvariety of $X\subset\PP^{2d+1}$ for which the $PGL_2$ action restricts, we use the lower star to indicate the intersection with the free locus: $X_*=X\cap\PP^{2d+1}_*$.  
\end{dfn}

It immediately follows from the definition that for any subvariety of the semistable locus, $X\subseteq\PPss$, $X_*$ is an open subvariety of $X$ and the restricted action of $\PGL_2$ on $X_*$ is free.

\begin{prop}
Let $d>1$ be an integer.  Then 
\begin{equation*}
\begin{split}
\Pic(\Ratd_*) &\cong\Pic(\Ratd), \\
\Pic(\PPs_*) &\cong\Pic(\PPs),\\ 
\Pic(\PPss_*) &\cong\Pic(\PPss).\\
\end{split}
\end{equation*}
\end{prop}

\begin{proof}
Since the varieties involved are non-singular, class group is isomorphic to Picard group, and the proof follows immediately using Proposition \ref{codimautproj} on the codimension of the automorphism locus in $\PP^{2d+1}$ and Proposition 6.5 of \cite{hartshorne}.
\end{proof}

We will use the following Descent lemma from \cite{drezet:picard} (Theoreme 2.3).   
\begin{thm}[Descent Lemma]\label{DescentLemma}
Let $\pi:X\to Y$ be a good quotient.  Let $F$ be a $G$-vector bundle on $X$.  Then $F$ descends to $Y$ if and only if for all closed points $x$ of $X$ such that the orbit $Gx$ is closed, the stabilizer of $x$ in $G$ acts trivially on the fiber $F_x$.
\end{thm}

\begin{cor}There is a natural injection 
\begin{equation*}
\pi^\star:\Pic (Y)\to\Pic^G (X),
\end{equation*}
whose image consists of all isomorphism classes of $G$-linearized line bundles $L$ on $X$ with the following property:

For every closed point $x$ in $X$ for which the orbit $Gx$ is closed, the stabilizer of $x$ in $G$ acts trivially on the fiber $L_x$.
\end{cor}

\begin{proof} It is easy to show that the map $\pi^\star: \Pic (Y)\to \Pic (X)$ is injective when $\pi$ is a good quotient. Furthermore, $\Pic(X)\to\Pic^G(X)$ is injective, proving the Corollary.
\end{proof}

\begin{prop}\label{UniqueLin}
Let $X$ be a normal irreducible variety over $k$ with $SL_2$-action.  Then any $\Lcal\in\Pic(X)$ admits a unique $SL_2$-linearization.  
\end{prop}

\begin{proof}
By Theorem 7.2 of \cite{dolgachev}, the natural map $\alpha: \Pic^{\SL_2}(X) \rightarrow \Pic(X)$ fits into an exact sequence
\begin{equation*}
0\rightarrow \Ker(\alpha)\rightarrow\Pic^{\SL_2}(X) \overset{\alpha}{\rightarrow} \Pic(X) \rightarrow \Pic(SL_2)
\end{equation*} 

$\Pic(\SL_2)$ is known to be trivial and consequently $\alpha$ is surjective.  $\Ker(\alpha)$ consists of the isomorphism classes of $\SL_2$-linearizations on the trivial bundle $X\times\AA^1$. In fact, by Corollary 7.1 of \cite{dolgachev}, $\Ker(\alpha)\cong\Hom(\SL_2,\GG_m)$, i.e. that $\Ker(\alpha)$ is isomorphic to the character group of $\SL_2$. The character group of a connected affine algebraic group which is complete is trivial, so $\Ker(\alpha)=1$ and the proposition is finished.
\end{proof}

\begin{thm}
In the notation above, there is a natural isomorphism between $\Ratd_*/\SL_2$ and $\mathcal{M}_{d,*}=\Md-\Acal$.
\end{thm}

\begin{proof}
This follows from basic facts of GIT.  $\Md$ is a geometric quotient.  $\Ratd_*$ is an open, $\SL_2$ invariant subset of $\Rat_d$, so it also has a geometric quotient. 

In general if $G$ acts on $X$ and $\Lcal\in\Pic^G(X)$, then if $U,V\subset X$ are open sets such that $V\subset U\subset X^s(\Lcal)$, then $V/G\subset U/G$ is an open subset of the quotient, both being geometric.

If follows that $\Ratd_*/\SL_2\subset \Md$ is an open subvariety where all points have trivial stabilizer.  Moreover, it is the largest such open subvariety.  It therefore coincides with $\mathcal{M}_{d,*}=\Md-\Acal$.
\end{proof}

\begin{prop} 
\
\begin{compactenum}
\item $\Cl(\Md)$ is as follows: $\ZZ/2d\ZZ$ when $d$ is odd and $\ZZ/d\ZZ$ when $d$ is even.  
\item $\Cl(\Md^s) = \Cl(\Md^{ss}) \cong \ZZ$
\end{compactenum}
\end{prop}
\begin{proof}
\textbf{(1)}  Note first that, because $\Ratd$ is a normal scheme, $\Md$ is also normal (see proposition 3.1 p. 45 of \cite{dolgachev}) and so Proposition 6.5 of \cite{hartshorne} applies. We can thus conclude that $\Cl(\Md)=\Cl(\mathcal{M}_{d*})$.  It follows from the previous theorems that $\mathcal{M}_{d*}$ is a good (in fact geometric) quotient of $\Ratd_*$ by $\SL_2$ and by Descent lemma (Theorem~\ref{DescentLemma}) the image of $\Pic(\Md)$ in $\Pic^{\SL_2}(\Ratd_*)$ consists of of isomorphism classes of linearized line bundles $L$ for which $<\pm I>$ acts trivially on the fiber of $L$ at every closed point of $\Ratd_*$. By lemma~\ref{eigenvalues} the the action of $-I$ on $\Ocal_{\PP^{2d+1}}(1)$ is trivial for $d$ odd and is via multiplication by $-1$ throughout when $d$ is even.
 
$\Ratd_*$ is an open subscheme of projective space and is therefore normal.  By Proposition \ref{UniqueLin} it follows that $\Pic^{\SL_2}(\Ratd_*)=\Pic(\Ratd_*) =\ZZ/2d\ZZ$.  And our calculation shows that the image of $\Pic(\Md)$ in $\Pic(\Ratd_*)$ is generated by $\Ocal(1)$, when $d$ is odd and by $\Ocal(2)$ when $d$ is even. We conclude that $\Pic({\Md}_*)=\ZZ/2d\ZZ$ for odd $d$ and $\Pic({\Md}_*)=\ZZ/d\ZZ$ for even $d$.  Finally, since ${\Md}_*$ is non-singular (Proposition~\ref{propsingaut}), $\Pic({\Md}_*) = \Cl({\Md}_*)$, and since codimension of its complement in $\Md$ is greater than $1$, $\Cl({\Md}_*)= \Cl(\Md)$.

\textbf{(2)}  The proof is nearly identical for $\Mds$. This approach requires more care for $\Mdss$ as $\Mdss$ is not a geometric quotient.  Regardless, $\Mds$ is a dense open subset of $\Mdss$ and therefore $\Mds=\Mdss - V$ for some close subset of codimension $\codim(V,\Mdss)>1$.  It follows that $\Cl(\Mdss)=\Cl(\Mds)=\ZZ$.   
\end{proof}

\begin{thm}Let $d>1$ be an integer. Then $\Pic(\Md)$ is trivial.
\end{thm}
\begin{proof}
By the descent lemma (Theorem \ref{DescentLemma}), $\Pic(\Md)$ can be identified with those elements $L$ of 
\begin{equation*}
\Pic^{\SL_2}(\Ratd)= \Pic(\Ratd)= \ZZ \Ocal(1)/\ZZ \Ocal(2d)
\end{equation*}
for which the stabilizer of every function acts trivially on the corresponding fiber of $L$.  We know that $\Ratd$ contains functions $\phi$ with automorphisms $\sigma\in\PGL_2$ of order $m$ as long as $m|d$ or $m|d\pm 1$.   Let $g$ represent $\sigma$ in $\SL_2$.   Then $g$ acts on the stalk of $\Ocal(1)$ at $\phi$ by multiplication by $s$-th root of unity, where $s$ is described in Corollary~\ref{eigenRatd}.

\textbf{Case $d$ even.} Take $m=d$.  Then $s=2d$, hence only line bundles $\Ocal(2dk)$ descend to $\Md$, but these are all trivial.

\textbf{Case $d$ odd.}  Take $m=d$ again.  Then $s=d$.  Then take $m=d+1$ or $m=d-1$.  Then $s=2$.  Again, only line bundles $\Ocal(2dk)$ descend to $\Md$, and these are all trivial.
\end{proof}

\begin{lem}\label{stabstable}
Let $x\in\PPss$ be a closed point.  Then $x\in\PPs$ if and only if its orbit is closed in $\PPss$ and its stabilizer is of finite order. Moreover, the order of the stabilizer of $x$ is bounded as $x$ runs through $\PPs$.
\end{lem}
\begin{proof}
For the first part, see  (cf~\cite{mumfordGIT}, Amplification 1.11).   For the second, consider the quasi-finite maps
\[\SL_2\times\PPs\to \PPs\times\PPs\]
given by 
\[(g,x)\mapsto (x, gx).\]
The stabilizers then correspond to the fibers over the diagonal, and quasi-finite maps of varieties have bounded fiber cardinality.
\end{proof}

\begin{thm} Let $d>1$ be an integer. The open immersion $\Mds\subset\Mdss$ induces an isomorphism $\Pic\Md^{ss}\cong\Pic\Md^{s}\cong\ZZ$.
\end{thm}
\begin{proof}Recall that $\Pic(\PPs)=\Pic(\PPss)$ (cf Lemma~\ref{picard}) is infinite cyclic generated by $\Ocal(1)$.

Let us look at $\Pic(\Mds)$ first.  Denote by $M$ a common multiple of orders of stabilizers of points in $\PP^{2d+1s}$.   Such $M$ exists by lemma~\ref{stabstable}.  Then  $\Ocal( M)$ descends to the quotient by Descent lemma~\ref{DescentLemma} and since the Picard subgroup of the quotient can be identified with  line bundles which descend to the quotient, the conclusion follows.

Now let us look at $\Pic(\Mdss)$ and assume that $d=2m+1$, since for even $d$, $\Mdss=\Mds$.  Let $x\in \Pic(\PPss)\backslash\Pic(\PPs)$ have a closed orbit.  Then the stabilizer of $x$ is infinite by Lemma~\ref{stabstable}.  Direct computation using the numerical criterion (cf ~\cite{silverman:ads}, Theorem 4.40)  shows that $x$ is conjugate to either $(X^{m+1}Y^m:0)$ or $(cX^{m+1}Y^m: X^mY^{m+1})$ with $\{f(z)=cz|c\neq 0\}$ as the stabilizer within $\PGL_2$ in both cases.  Using Lemma~\ref{eigenvalues} one can see that in both cases the stabilizer within $\SL_2$ acts trivially on the fiber of the line bundle $\Ocal(1)$ at $x$.  Hence such a point $x$ does not obstruct the descent of line bundles and thus a line bundle over $\PPss$ descends to the quotient if and only if its restriction to $\PPs$ does.
\end{proof}


\medskip

\bibliographystyle{acm}

\medskip

\end{document}